\documentclass[11pt]{article}
\usepackage{amsfonts,amscd,amssymb, amsmath}
\usepackage[usenames]{color}

\usepackage{graphicx}

\usepackage{euscript}
\usepackage[all,tips]{xy}
\usepackage{epic,eepic}

\usepackage{color}
\usepackage{tikz}

\newtheorem{definition}{Definition}[section]
\newtheorem{lemma}[definition]{Lemma}
\newtheorem{proposition}[definition]{Proposition}
\newtheorem{corollary}[definition]{Corollary}
\newtheorem{remark}[definition]{Remark}

\def\rawo\lonra{\longrightarrow}

\def\ot{\otimes}

\allowdisplaybreaks[4]
\newenvironment{proof}{{\it Proof.}}{\hfill $ \square $ \vskip 4mm}

\begin{document}

\title{BiHom-pre-Lie algebras, BiHom-Leibniz algebras and Rota-Baxter operators on BiHom-Lie algebras
}
\author{Ling Liu \\
College of Mathematics and Computer Science,\\
Zhejiang Normal University, \\
Jinhua 321004, China \\
e-mail: ntliulin@zjnu.cn \and Abdenacer Makhlouf \\
Universit\'{e} de Haute Alsace, \\
IRIMAS-D\'epartement  de Math\'{e}matiques,  \\
6 bis, rue des fr\`{e}res Lumi\`{e}re, F-68093 Mulhouse, France\\
e-mail: Abdenacer.Makhlouf@uha.fr \and Claudia Menini \\
University of Ferrara, \\
Department of Mathematics and Computer Science\\
Via Machiavelli 30, Ferrara, I-44121, Italy\\
e-mail: men@unife.it \and Florin Panaite \\
Institute of Mathematics of the Romanian Academy,\\
PO-Box 1-764, RO-014700 Bucharest, Romania\\
e-mail: florin.panaite@imar.ro }
\maketitle

\begin{abstract}
We contribute to the study of Rota-Baxter operators on other types of algebras than associative and Lie algebras.
If $A$ is an algebra of a certain type and $R$ is a Rota-Baxter operator on $A$, one can
define a new multiplication on $A$ by means of $R$ and the previous multiplication and ask under
what circumstances the new algebra is of the same type as $A$. Our first main result deals with such a situation in the case of BiHom-Lie algebras. Our second main result is a BiHom analogue
of Aguiar's theorem that shows how to obtain a pre-Lie algebra from a Rota-Baxter operator of weight zero on a Lie algebra.
The BiHom analogue does not work for BiHom-Lie algebras, but for a new concept we introduce here, called
left BiHom-Lie algebra, at which we arrived by defining first the BiHom version of Leibniz algebras. \\
\textbf{Keywords}: Rota-Baxter operator; BiHom-Lie algebra; BiHom-Leibniz algebra\\
\textbf{MSC2010}: 15A04; 17A99; 17D99
\end{abstract}

\section{Introduction}
Algebras of Hom-type appeared in the Physics literature of the 1990's, in the context of quantum deformations
of some algebras of vector fields, such as the Witt and Virasoro algebras, in connection with oscillator algebras
(\cite{AizawaSaito,Hu}).
A quantum deformation consists of replacing the usual derivation by a $\sigma$-derivation.
It turns out that the algebras obtained in this way do not satisfy the Jacobi identity anymore, but instead they satisfy a
modified version involving a homomorphism. This kind of algebras
were called Hom-Lie algebras and studied by Hartwig, Larsson and Silvestrov in \cite{JDS,DS}. The Hom analogue of
associative algebras was introduced  in \cite{ms}, where it is shown
that the commutator bracket defined by the multiplication in a Hom-associative
algebra leads to a Hom-Lie algebra. A categorical approach to Hom-type algebras was considered  in \cite{stef}.  A generalization  has been given in \cite{gmmp}, where a construction of a Hom-category including a group action led to concepts of BiHom-type algebras. Hence, BiHom-associative algebras and BiHom-Lie algebras, involving two linear maps (called structure maps), were introduced. The main axioms for these types of algebras
(BiHom-associativity, BiHom-skew-symmetry and BiHom-Jacobi condition) were dictated by categorical considerations.

Rota-Baxter operators first appeared in G. Baxter's work in probability and the study of fluctuation theory (\cite{Baxter}).  Afterwards, Rota-Baxter algebras were intensively studied by Rota (\cite{Rota,Rota2}) in connection with combinatorics.
Rota-Baxter operators have appeared in a wide range of areas in pure and applied mathematics, for example in
the work of  Connes and Kreimer (\cite{Connes-Kreimer}) about their Hopf algebra approach to renormalization  of quantum field theory. This seminal work was followed by an important development of the theory of
Rota-Baxter algebras and their connections to other algebraic structures, see for example the book \cite{Guo} and
references therein. In the context of Lie algebras, Rota-Baxter operators were introduced
independently by Belavin and Drinfeld (\cite{Drinfeld}) and Semenov-Tian-Shansky (\cite{semenov}), in connection with solutions of the (modified) classical Yang-Baxter equation.

The first aim of this paper is to obtain a BiHom analogue
of the classical result, due to Aguiar (\cite{aguiar}), saying that, if
$(L, [\cdot , \cdot ])$ is a Lie algebra and $R:L\rightarrow L$ is a Rota-Baxter operator of weight $0$, and one defines a new
operation on $L$ by $x\cdot y=[R(x), y]$, then $(L, \cdot )$ is a left pre-Lie algebra. The Hom analogue of this
result was obtained in \cite{makhloufrotabaxter}, see also \cite{mak-yau}. So, we first define, in Section \ref{sec3}, the concepts of
(left and right) BiHom-pre-Lie algebras
and present some properties of these objects. Then, our aim was to prove that, if $(L, [\cdot , \cdot ], \alpha , \beta )$
is a BiHom-Lie algebra and $R:L\rightarrow L$ is a Rota-Baxter operator of weight $0$ such that
$R\circ \alpha =\alpha \circ R$ and $R\circ \beta =\beta \circ R$, then $(L, \cdot, \alpha , \beta )$
is a left BiHom-pre-Lie algebra, where $x\cdot y=[R(x), y]$. It turns out that this does not work
(unless the structure maps $\alpha $
and $\beta $ are bijective). We were thus led to realize that, apart from the concept of BiHom-Lie algebra
(introduced in \cite{gmmp}), there exist other natural BiHom analogues of Lie and Hom-Lie algebras,
that we called left and right BiHom-Lie algebras. We arrived at these concepts as follows (we concentrate on
the left case): we introduced first the concept of
left BiHom-Leibniz algebra (the natural BiHom analogue of Loday's concept of left Leibniz algebra)
and then a left BiHom-Lie algebra is a left BiHom-Leibniz algebra satisfying the
BiHom-skew-symmetry condition. It turns out (Proposition \ref{equivBHLie}) that the concepts of BiHom-Lie algebra and
left BiHom-Lie algebra coincide if the structure maps are bijective, and that the BiHom analogue of Aguiar's result
mentioned above holds indeed for left BiHom-Lie algebras (Proposition
\ref{lrprelie}), although it does not hold in general for BiHom-Lie algebras.

In a previous paper (\cite{lmmp1}), where we studied Rota-Baxter operators on BiHom-associative algebras,
we proved the following result. Let $(A, \cdot , \alpha , \beta)$ be a BiHom-associative algebra,
$R:A\rightarrow A$ a Rota-Baxter operator of weight $\lambda $ commuting with $\alpha $ and $\beta $, and
define a new multiplication on $A$ by $a*b=R(a)\cdot b+a\cdot R(b)+\lambda a\cdot b$, for all $a, b\in A$; then
$(A, * , \alpha , \beta)$ is also a BiHom-associative algebra. The second aim of the current paper is to prove a similar result for BiHom-Lie algebras; this is achieved in Proposition
\ref{Proposition:BiHLie}.  We also prove similar results for left or right BiHom-Leibniz algebras and for left or right
BiHom-Lie algebras.

In Section \ref{sec6} we prove that, in the case of bijective structure maps,
Propositions \ref{lrprelie} and \ref{Proposition:BiHLie} admit a common generalization, in terms of what we call
BiHom-PostLie algebras (a concept introduced independently in \cite{adimi}),
the BiHom analogue of the concept of PostLie algebra introduced in \cite{vallette} and
studied in \cite{bai}.

\section{Preliminaries}\label{sec1}
\setcounter{equation}{0} 

All algebraic structures in this paper
(algebras, linear spaces etc.) will be defined over a base field $\Bbbk $; we denote
$\otimes_{\Bbbk}$ simply by $\otimes $. We denote by $_{\Bbbk }\mathcal{M}$ the category of linear spaces over $\Bbbk $.
Unless otherwise specified, the
algebras we work with are
\emph{not} required to be unital, and a multiplication $\mu :V\otimes
V\rightarrow V$ on a linear space $V$ is denoted by $\mu
(v\otimes v^{\prime })=vv^{\prime }$. We write $g\circ f$ or simply $gf$ for the composition of two maps $f$ and $g$.
The identity
map on a linear space $V$ is denoted by $id_V$.
We denote by $\circlearrowleft _{x, y, z}$ summation over the cyclic permutations of some elements
$x, y, z$.

\begin{definition}
A left (respectively right) pre-Lie algebra is a linear space $A$ endowed with a linear map
$\cdot :A\ot A\rightarrow A$ satisfying, for all $x, y, z\in A$,
\begin{eqnarray}
&&x\cdot (y\cdot z)-(x\cdot y)\cdot z=y\cdot (x\cdot z)-(y\cdot x)\cdot z,
\label{lpL}
\end{eqnarray}
respectively
\begin{eqnarray}
&&x\cdot (y\cdot z)-(x\cdot y)\cdot z=x\cdot (z\cdot y)-(x\cdot z)\cdot y.
\label{rpL}
\end{eqnarray}
\end{definition}

Any associative algebra is a left and a right pre-Lie algebra. If $(A, \cdot )$ is a left or a right pre-Lie
algebra, then $(A, [x, y]=x\cdot y-y\cdot x)$ is a Lie algebra.
\begin{definition} (\cite{lodayleibniz})
A left (respectively right) Leibniz algebra is a pair $(L, [\cdot ,\cdot  ])$, where $L$ is a linear space
and
$[\cdot ,\cdot  ]:L\times L\rightarrow L$ is a bilinear map satisfying, for all $x, y, z\in L$,
\begin{eqnarray}
&&[x, [y, z]]=[[x, y], z]+[y, [x, z]], \label{leftleibniz}
\end{eqnarray}
respectively
\begin{eqnarray}
&&[[x, y], z]=[[x, z], y]+[x, [y, z]]. \label{rightleibniz}
\end{eqnarray}
A morphism of (left or right) Leibniz algebras $L$ and $L'$ is a linear map $f:L\rightarrow L'$ such that
$f([x, y])=[f(x), f(y)]$, for all $x, y\in L$.
\end{definition}

A Rota-Baxter structure on an algebra of a given type is defined as follows.
\begin{definition}\label{RB-DEF}
Let $A$ be a linear space and $\mu :A\ot A\rightarrow A$, $\mu (x\ot y)=xy$, for all $x, y\in A$, a
linear multiplication on $A$ and let $\lambda \in \Bbbk $. A Rota-Baxter operator of weight $\lambda $
for $(A, \mu )$ is a linear map $R:A\rightarrow A$ satisfying the so-called Rota-Baxter condition
\begin{eqnarray}
&&R(x)R(y)=R(R(x)y+xR(y)+\lambda xy), \;\;\;\forall \; x, y\in A. \label{RBrel}
\end{eqnarray}
\end{definition}

In this case, if we define on $A$ a new multiplication by $x*y=xR(y)+R(x)y+\lambda xy$,
for all $x, y\in A$, then
$R(x*y)=R(x)R(y)$, for all $x, y\in A$, and $R$ is a Rota-Baxter operator of weight $\lambda $ for
$(A, *)$. If $(A, \mu )$ is associative, then $(A, *)$ is also associative.

We recall now from \cite{gmmp} several facts about BiHom-type algebras.
\begin{definition}
A BiHom-associative algebra is a 4-tuple $\left( A,\mu ,\alpha ,\beta \right) $, where $A$ is
a linear space, $\alpha :A\rightarrow A$, $\beta :A\rightarrow A$
and $\mu :A\otimes A\rightarrow A$ are linear maps, with notation $\mu (x\otimes y) =xy$, for all $x, y\in A$, satisfying the following
conditions, for all $x, y, z\in A$:
\begin{gather}
\alpha \circ \beta =\beta \circ \alpha , \\
\alpha (xy) =\alpha (x)\alpha (y) \text{ and }\beta (xy)=\beta (x)\beta (y) ,\quad \text{%
(multiplicativity)}  \label{eqalfabeta} \\
\alpha (x)(yz)=(xy)\beta (z).\quad \text{%
(BiHom-associativity)}  \label{eqasso}
\end{gather}

The maps $\alpha $ and $\beta $ (in this order) are called the structure maps
of $A$.
\end{definition}
\begin{definition}\label{Def-BiHomLie}
A BiHom-Lie algebra $\left( L,
\left[\cdot , \cdot \right] ,\alpha ,\beta \right) $ is a 4-tuple in which $L$ is a linear
space, $\alpha , \beta :L\rightarrow L$ are linear maps and $\left[\cdot , \cdot \right]
:L\times L\rightarrow L$ is a bilinear map,
such that, for all $x, y, z\in L:$%
\begin{gather*}
\alpha \circ \beta =\beta \circ \alpha , \\
\alpha (\left[ x, y \right])=\left[ \alpha \left(
x\right),\alpha (y) \right]\;\;
\text{ and } \;\;\beta (\left[x, y \right])=\left[
\beta \left( x \right) ,\beta \left( y \right) %
\right], \\
\left[ \beta \left( x\right) ,\alpha \left( y\right) \right] =-%
\left[ \beta \left( y\right) ,\alpha \left( x\right) \right],
\;\;\;\; \text{ (BiHom-skew-symmetry)} \\
\left[ \beta ^{2}\left( x\right) ,\left[ \beta \left( y\right)
,\alpha \left( z\right) \right] \right] +\left[ \beta
^{2}\left( y\right) ,\left[ \beta \left( z\right)
,\alpha \left( x\right) \right] \right] +\left[ \beta ^{2}\left( z\right) ,\left[ \beta \left( x\right) ,\alpha \left( y
\right) \right] \right] =0. \\
\text{ (BiHom-Jacobi condition)}
\end{gather*}

The maps $\alpha $ and $\beta $ (in this order) are called the structure maps
of $L$.

A morphism between two BiHom-Lie algebras $\left( L,\left[\cdot , \cdot \right] ,\alpha ,\beta \right)$ and
$\left( L^{\prime },\left[\cdot , \cdot \right]^{\prime },\alpha ^{\prime },\beta
^{\prime }\right)$  is a linear map $f:L\rightarrow
L^{\prime }$ satisfying the conditions $\alpha ^{\prime }\circ f=f\circ \alpha $, $\beta
^{\prime }\circ f=f\circ \beta $ and $f(\left[x, y\right])=[f(x),
f(y)]^{\prime }$, for all $x, y\in L$.
\end{definition}

If $\left( L,\left[\cdot , \cdot \right] \right) $ is an ordinary Lie algebra
and $\alpha ,\beta :L\rightarrow L$ are two commuting morphisms of Lie algebras,  and
we define the bilinear map $\left\{\cdot , \cdot \right\} :L\times L\rightarrow L$,
$\left\{ x, y\right\} =\left[ \alpha \left( x\right) ,\beta \left( y\right) %
\right]$, for all $x, y\in L$,
then $L_{\left( \alpha ,\beta \right) }=(L, \left\{\cdot , \cdot \right\}, \alpha ,
\beta )$ is a BiHom-Lie algebra, called the Yau twist of $\left( L,%
\left[\cdot , \cdot \right] \right) $.

If $\left( L,\left[\cdot , \cdot \right] ,\alpha ,\beta \right) $ is a
BiHom-Lie algebra, $\alpha^{\prime }, \beta^{\prime }:L\rightarrow L$ are
morphisms of BiHom-Lie algebras and any two
of the maps $\alpha, \beta , \alpha ^{\prime }, \beta ^{\prime }$ commute,
then $\left( L,\left[\cdot , \cdot
\right]\circ(\alpha^{\prime }\otimes\beta^{\prime }),
\alpha\circ\alpha^{\prime },\beta\circ\beta^{\prime }\right) $ is a
BiHom-Lie algebra.
\section{BiHom-pre-Lie algebras}\label{sec3}
\setcounter{equation}{0}
\begin{definition}
A left (respectively right) BiHom-pre-Lie algebra $(A, \cdot , \alpha , \beta )$ is a 4-tuple in which $A$ is
a linear space and
$\cdot :A\ot A\rightarrow A$ and $\alpha , \beta :A\rightarrow A$ are linear maps satisfying
$\alpha \circ \beta =\beta \circ \alpha $, $\alpha (x\cdot y)=\alpha (x)\cdot \alpha (y)$,
$\beta (x\cdot y)=\beta (x)\cdot \beta (y)$ and
\begin{eqnarray}
&&\alpha \beta (x)\cdot (\alpha (y)\cdot z)-(\beta (x)\cdot \alpha (y))\cdot \beta (z)
=\alpha \beta (y)\cdot (\alpha (x)\cdot z)-(\beta (y)\cdot \alpha (x))\cdot \beta (z),
\label{lBHpL}
\end{eqnarray}
respectively
\begin{eqnarray}
&&\alpha (x)\cdot (\beta (y)\cdot \alpha (z))-(x\cdot \beta (y))\cdot \alpha \beta (z)=
\alpha (x)\cdot (\beta (z)\cdot \alpha (y))-(x\cdot \beta (z))\cdot \alpha \beta (y),
\label{rBHpL}
\end{eqnarray}
for all $x, y, z\in A$. The maps $\alpha $ and $\beta $ (in this order) are called the structure maps
of $A$.

A morphism between two left or right BiHom-pre-Lie algebras $(A, \cdot , \alpha , \beta )$ and
$(A', \cdot ', \alpha ', \beta ')$
is a linear map
$f:A\rightarrow A'$ satisfying $f(x\cdot y)=f(x)\cdot ' f(y)$, for all $x, y\in A$,
as well as $f\circ \alpha =\alpha '\circ f$ and $f\circ \beta =\beta '\circ f$.
\end{definition}

Obviously, any BiHom-associative algebra is also a
left and a right BiHom-pre-Lie algebra.

If $(A, \cdot , \alpha , \beta )$ is a left BiHom-pre-Lie algebra and we define a new multiplication $*$ on $A$
by $x*y=y\cdot x$, then $(A, * , \beta , \alpha )$ is a right BiHom-pre-Lie algebra.
\begin{proposition}
If $(A, \cdot )$ is a left (respectively right) pre-Lie algebra and $\alpha , \beta :A\rightarrow A$ are linear maps satisfying
$\alpha \circ \beta =\beta \circ \alpha $, $\alpha (x\cdot y)=\alpha (x)\cdot \alpha (y)$ and
$\beta (x\cdot y)=\beta (x)\cdot \beta (y)$, for all $x, y\in A$, and we define a new multiplication on $A$ by
$x*y=\alpha (x)\cdot \beta (y)$, then $(A, * , \alpha , \beta )$ is a left (respectively right) BiHom-pre-Lie algebra,
called the Yau twist of $(A, \cdot )$.
\end{proposition}
\begin{proof}
We only prove (\ref{lBHpL}) and leave the rest to the reader:\\[2mm]
${\;\;\;}$$\alpha \beta (x)* (\alpha (y)* z)-(\beta (x)* \alpha (y))* \beta (z)$
\begin{eqnarray*}
&=&\alpha ^2\beta (x)\cdot (\alpha ^2\beta (y)\cdot \beta ^2(z))-(\alpha ^2\beta (x)\cdot
\alpha ^2\beta (y))\cdot \beta ^2(z)\\
&\overset{(\ref{lpL})}{=}&\alpha ^2\beta (y)\cdot (\alpha ^2\beta (x)\cdot \beta ^2(z))-(\alpha ^2\beta (y)\cdot
\alpha ^2\beta (x))\cdot \beta ^2(z)\\
&=&\alpha \beta (y)* (\alpha (x)* z)-(\beta (y)* \alpha (x))* \beta (z),
\end{eqnarray*}
finishing the proof.
\end{proof}
\begin{remark}
More generally, one can prove that, if
$(A, \cdot ,\alpha , \beta )$ is a left (respectively right) BiHom-pre-Lie algebra and $\tilde{\alpha }, \tilde{\beta }:
A\rightarrow A$ are two morphisms of BiHom-pre-Lie algebras such that any two of the maps $\alpha, \beta ,
\tilde{\alpha }, \tilde{\beta }$ commute, and we define a new multiplication on $A$ by
$x*y=\tilde{\alpha }(x)\cdot \tilde{\beta }(y)$,
for all $x, y\in A$, then $(A, *, \alpha \circ \tilde{\alpha }, \beta \circ \tilde{\beta })$
is also a left (respectively right) BiHom-pre-Lie algebra.
\end{remark}
\begin{proposition}
Let $(A, \cdot , \alpha , \beta )$ be a left or a right BiHom-pre-Lie algebra such that $\alpha $ and $\beta $ are bijective.
Define $[\cdot ,\cdot  ]:A\ot A\rightarrow A$ by $[x, y]=x\cdot y-(\alpha ^{-1}\beta (y))\cdot (\alpha \beta ^{-1}(x))$,
for all $x, y\in A$. Then $(A, [\cdot ,\cdot  ] , \alpha , \beta )$ is a BiHom-Lie algebra.
\end{proposition}
\begin{proof}
We only give the proof in the case $A$ is a left BiHom-pre-Lie algebra, the other case is similar and left to the reader.
The fact that $\alpha $ and $\beta $ are multiplicative with respect to $[\cdot ,\cdot  ]$ is obvious, and so is the
BiHom-skew-symmetry relation.
We only have to prove that BiHom-Jacobi condition holds. We compute, for
$x, y, z\in A$,
\begin{eqnarray*}
\circlearrowleft _{x, y, z}[\beta ^2(x), [\beta (y), \alpha (z)]]&=&
\circlearrowleft _{x, y, z}[\beta ^2(x), \beta (y)\cdot \alpha (z)-\beta (z)\cdot \alpha (y)]\\
&=&\circlearrowleft _{x, y, z}([\beta ^2(x), \beta (y)\cdot \alpha (z)]-[\beta ^2(x), \beta (z)\cdot \alpha (y)])\\
&=&\circlearrowleft _{x, y, z}(\beta ^2(x)\cdot (\beta (y)\cdot \alpha (z))-\alpha ^{-1}\beta (\beta (y)\cdot \alpha (z))\cdot
\alpha \beta (x)\\
&&-\beta ^2(x)\cdot (\beta (z)\cdot \alpha (y))+\alpha ^{-1}\beta (\beta (z)\cdot \alpha (y))\cdot
\alpha \beta (x))\\
&=&\circlearrowleft _{x, y, z}(\beta ^2(x)\cdot (\beta (y)\cdot \alpha (z))-(\alpha ^{-1}\beta ^2(y)\cdot \beta (z))\cdot
\alpha \beta (x)\\
&&-\beta ^2(x)\cdot (\beta (z)\cdot \alpha (y))+(\alpha ^{-1}\beta ^2(z)\cdot \beta (y))\cdot
\alpha \beta (x))\\
&=&\beta ^2(x)\cdot (\beta (y)\cdot \alpha (z))-(\alpha ^{-1}\beta ^2(y)\cdot \beta (z))\cdot
\alpha \beta (x)\\
&&-\beta ^2(x)\cdot (\beta (z)\cdot \alpha (y))+(\alpha ^{-1}\beta ^2(z)\cdot \beta (y))\cdot
\alpha \beta (x)\\
&&+\beta ^2(y)\cdot (\beta (z)\cdot \alpha (x))-(\alpha ^{-1}\beta ^2(z)\cdot \beta (x))\cdot
\alpha \beta (y)\\
&&-\beta ^2(y)\cdot (\beta (x)\cdot \alpha (z))+(\alpha ^{-1}\beta ^2(x)\cdot \beta (z))\cdot
\alpha \beta (y)\\
&&+\beta ^2(z)\cdot (\beta (x)\cdot \alpha (y))-(\alpha ^{-1}\beta ^2(x)\cdot \beta (y))\cdot
\alpha \beta (z)\\
&&-\beta ^2(z)\cdot (\beta (y)\cdot \alpha (x))+(\alpha ^{-1}\beta ^2(y)\cdot \beta (x))\cdot
\alpha \beta (z)\\
&\overset{(\ref{lBHpL})\; 3\; times}{=}&0+0+0=0,
\end{eqnarray*}
finishing the proof.
\end{proof}
\begin{definition} (\cite{lmmp1}) A BiHom-dendriform algebra is a 5-tuple $(A, \prec , \succ , \alpha , \beta )$
consisting of a linear space $A$ and linear maps $\prec , \succ :A\otimes A\rightarrow A$ and
$\alpha , \beta :A\rightarrow A$
satisfying, for all $x, y, z\in A$,
the following conditions:
\begin{eqnarray}
&&\alpha \circ \beta =\beta \circ \alpha , \label{BiHomdend1} \\
&&\alpha (x\prec y)=\alpha (x)\prec \alpha (y), ~
\alpha (x\succ y)=\alpha (x)\succ \alpha (y), \label{BiHomdend3} \\
&&\beta (x\prec y)=\beta (x)\prec \beta (y), ~
\beta (x\succ y)=\beta (x)\succ \beta (y), \label{BiHomdend5} \\
&&(x\prec y)\prec \beta (z)=\alpha (x)\prec (y\prec z+y\succ z),  \label{BiHomdend6} \\
&&(x \succ y)\prec \beta (z)=\alpha (x)\succ (y\prec z), \label{BiHomdend7} \\
&&\alpha (x)\succ (y\succ z)=(x\prec y+x\succ y)\succ \beta (z). \label{BiHomdend8}
\end{eqnarray}

We call $\alpha $ and $\beta $ (in this order) the structure maps
of $A$.
\end{definition}
\begin{proposition} \label{BHdendpreLie}
Let $(A, \prec , \succ , \alpha , \beta )$ be a BiHom-dendriform algebra such that $\alpha $ and $\beta $
are bijective. Let $\rhd , \lhd :A\ot A\rightarrow A$ be linear maps defined for all $x, y\in A$ by
\begin{eqnarray*}
&&x\rhd y=x\succ y-(\alpha ^{-1}\beta (y))\prec (\alpha \beta ^{-1}(x)), \;\;\;\;\;
x\lhd y=x\prec y-(\alpha ^{-1}\beta (y))\succ (\alpha \beta ^{-1}(x)).
\end{eqnarray*}
Then $(A, \rhd , \alpha , \beta )$  (respectively $(A, \lhd , \alpha , \beta )$)
is a left (respectively right) BiHom-pre-Lie algebra.
\end{proposition}
\begin{proof}
We only check  identity (\ref{lBHpL}) and leave the rest to the reader. We compute:\\[2mm]
${\;\;\;}$$\alpha \beta (x)\rhd (\alpha (y)\rhd z)-(\beta (x)\rhd \alpha (y))\rhd \beta (z)$
\begin{eqnarray*}
&=&\alpha \beta (x)\rhd (\alpha (y)\succ z-\alpha ^{-1}\beta (z)\prec \alpha ^2\beta ^{-1}(y)) \\
&&-(\beta (x)\succ \alpha (y)-\beta (y)\prec \alpha (x))\rhd \beta (z)\\
&=&\alpha \beta (x)\succ (\alpha (y)\succ z)-(\beta (y)\succ \alpha ^{-1}\beta (z))\prec \alpha ^2(x)\\
&&-\alpha\beta (x)\succ (\alpha ^{-1}\beta (z)\prec \alpha ^2\beta ^{-1}(y))+
(\alpha ^{-2}\beta ^2(z)\prec \alpha (y))\prec \alpha ^2(x)\\
&&-(\beta (x)\succ \alpha (y))\succ \beta (z)+\alpha ^{-1}\beta ^2(z)\prec (\alpha (x)\succ
\alpha ^2\beta ^{-1}(y))\\
&&+(\beta (y)\prec \alpha (x))\succ \beta (z)-\alpha ^{-1}\beta ^2(z)\prec (\alpha (y)\prec
\alpha ^2\beta ^{-1}(x))\\
&\overset{(\ref{BiHomdend6}), \;(\ref{BiHomdend8})}{=}&
(\beta (x)\prec \alpha (y))\succ \beta (z)+(\beta (x)\succ \alpha (y))\succ \beta (z)\\
&&-(\beta (y)\succ \alpha ^{-1}\beta (z))\prec \alpha ^2(x)
-\alpha\beta (x)\succ (\alpha ^{-1}\beta (z)\prec \alpha ^2\beta ^{-1}(y))\\
&&+
\alpha ^{-1}\beta ^2(z)\prec (\alpha (y)\prec \alpha ^2\beta ^{-1}(x))
+
\alpha ^{-1}\beta ^2(z)\prec (\alpha (y)\succ \alpha ^2\beta ^{-1}(x))\\
&&-(\beta (x)\succ \alpha (y))\succ \beta (z)+\alpha ^{-1}\beta ^2(z)\prec (\alpha (x)\succ
\alpha ^2\beta ^{-1}(y))\\
&&+(\beta (y)\prec \alpha (x))\succ \beta (z)-\alpha ^{-1}\beta ^2(z)\prec (\alpha (y)\prec
\alpha ^2\beta ^{-1}(x))\\
&=&(\beta (x)\prec \alpha (y))\succ \beta (z)
+(\beta (y)\prec \alpha (x))\succ \beta (z)\\
&&-(\beta (y)\succ \alpha ^{-1}\beta (z))\prec \alpha ^2(x)
-\alpha\beta (x)\succ (\alpha ^{-1}\beta (z)\prec \alpha ^2\beta ^{-1}(y))\\
&&+\alpha ^{-1}\beta ^2(z)\prec (\alpha (y)\succ \alpha ^2\beta ^{-1}(x))
+\alpha ^{-1}\beta ^2(z)\prec (\alpha (x)\succ
\alpha ^2\beta ^{-1}(y))\\
&\overset{(\ref{BiHomdend7})}{=}&(\beta (x)\prec \alpha (y))\succ \beta (z)
+(\beta (y)\prec \alpha (x))\succ \beta (z)\\
&&-(\beta (y)\succ \alpha ^{-1}\beta (z))\prec \alpha ^2(x)
-(\beta (x)\succ \alpha ^{-1}\beta (z))\prec \alpha ^2(y)\\
&&+\alpha ^{-1}\beta ^2(z)\prec (\alpha (y)\succ \alpha ^2\beta ^{-1}(x))
+\alpha ^{-1}\beta ^2(z)\prec (\alpha (x)\succ
\alpha ^2\beta ^{-1}(y)).
\end{eqnarray*}
This expression is obviously symmetric in $x$ and $y$, so we are done.
\end{proof}
\section{BiHom-Leibniz algebras and BiHom-Lie algebras}\label{sec5}
\setcounter{equation}{0}
\begin{definition}
A left (respectively right) BiHom-Leibniz algebra is a 4-tuple $(L, [\cdot ,\cdot ], \alpha , \beta )$, where $L$ is a
linear space,
$[\cdot , \cdot ]:L\times L\rightarrow L$ is a bilinear map and $\alpha , \beta :L\rightarrow L$ are linear maps
satisfying $\alpha \circ \beta =\beta \circ \alpha $, $\alpha ([x, y])=[\alpha (x), \alpha (y)]$,
$\beta ([x, y])=[\beta (x), \beta (y)]$ and
\begin{eqnarray}
&&[\alpha \beta (x), [y, z]]=[[\beta (x), y], \beta (z)]+[\beta (y), [\alpha (x), z]], \label{leftBHleibniz}
\end{eqnarray}
respectively
\begin{eqnarray}
&&[[x, y], \alpha \beta (z)]=[[x, \beta (z)], \alpha (y)]+[\alpha (x), [y, \alpha (z)]],  \label{rightBHleibniz}
\end{eqnarray}
for all $x, y, z\in L$. We call $\alpha $ and $\beta $ (in this order) the structure maps of $L$.

A morphism $f:( L,[\cdot ,\cdot ] ,\alpha ,\beta )\rightarrow
( L', [\cdot ,\cdot ]',\alpha ',\beta ')$ of BiHom-Leibniz algebras is a linear map $f:L\rightarrow
L'$ such that $\alpha '\circ f=f\circ \alpha $, $\beta '\circ f=f\circ \beta $ and $f([x, y])=[f(x),
f(y)]'$, for all $x, y\in L$.
\end{definition}
\begin{proposition}
If $(L, [\cdot ,\cdot ])$ is a left (respectively right) Leibniz algebra and $\alpha , \beta :L\rightarrow L$ are two commuting
morphisms of Leibniz algebras, and we define the map $\{ \cdot , \cdot \}:L\times L\rightarrow L$,
$\{x, y\}=[\alpha (x), \beta (y)]$, for all $x, y\in L$, then $(L, \{ \cdot , \cdot \}, \alpha , \beta )$ is a left (respectively right)
BiHom-Leibniz algebra, called the Yau twist of $L$ and denoted by $L_{(\alpha , \beta )}$.
\end{proposition}
\begin{proof}
We prove the case when $L$ is a left Leibniz algebra, the other case is similar and left to the reader. We compute:
\begin{eqnarray*}
\{\alpha \beta (x), \{y, z\}\}&=&\{\alpha \beta (x), [\alpha (y), \beta (z)]\}
=[\alpha ^2\beta (x), [\alpha \beta (y), \beta ^2(z)]],
\end{eqnarray*}
${\;\;\;\;\;}$$\{\{\beta (x), y\}, \beta (z)\}+\{\beta (y), \{\alpha (x), z\}\}$
\begin{eqnarray*}
&=&\{[\alpha \beta (x), \beta (y)], \beta (z)\}+\{\beta (y), [\alpha ^2(x), \beta (z)]\}\\
&=&[[\alpha ^2\beta (x), \alpha \beta (y)], \beta ^2(z)]+[\alpha \beta (y), [\alpha ^2\beta (x), \beta ^2(z)]],
\end{eqnarray*}
and the desired equality follows by applying (\ref{leftleibniz}) to the elements $\alpha ^2\beta (x)$,
$\alpha \beta (y)$, $\beta ^2(z)$.
\end{proof}
\begin{remark}
More generally, let $(L, [\cdot ,\cdot ], \alpha , \beta )$ be a left (respectively right) BiHom-Leibniz algebra and
$\alpha ', \beta ':L\rightarrow L$ two morphisms of BiHom-Leibniz algebras such that any two of the
maps $\alpha , \alpha ', \beta , \beta '$ commute. If we define the map $\{\cdot ,\cdot \}:L\times L\rightarrow L$ by
$\{x, y\}=[\alpha '(x), \beta '(y)]$, for all $x, y\in L$, then $(L, \{ \cdot , \cdot \}, \alpha \circ \alpha ', \beta \circ \beta ')$
is a left (respectively right)
BiHom-Leibniz algebra.
\end{remark}
\begin{remark}\label{leibnizleftright}
One can easily see that, if $(L, [\cdot ,\cdot ], \alpha , \beta )$ is a left (respectively right) BiHom-Leibniz algebra and
we define $\{ \cdot , \cdot \}:L\times L\rightarrow L$ by  $\{x, y\}:=[y, x]$, then $(L, \{ \cdot , \cdot \}, \beta , \alpha  )$ is a
right (respectively left) BiHom-Leibniz algebra.
\end{remark}
\begin{definition}
A left (respectively right) BiHom-Lie algebra is a left (respectively right) BiHom-Leibniz algebra
$(L, [\cdot ,\cdot ], \alpha , \beta )$ satisfying the BiHom-skew-symmetry condition
\begin{eqnarray}
&&[\beta (x), \alpha (y)]=-[\beta (y), \alpha (x)], \;\;\;\forall \;x, y\in L. \label{BHskewsym}
\end{eqnarray}
A morphism $f:( L,[\cdot ,\cdot ] ,\alpha ,\beta )\rightarrow
( L', [\cdot ,\cdot ]',\alpha ',\beta ')$ of BiHom-Lie algebras is a linear map $f:L\rightarrow
L'$ such that $\alpha '\circ f=f\circ \alpha $, $\beta '\circ f=f\circ \beta $ and $f([x, y])=[f(x),
f(y)]'$, for all $x, y\in L$.
\end{definition}
\begin{remark}
In view of Remark \ref{leibnizleftright}, if $(L, [\cdot ,\cdot ], \alpha , \beta )$ is a left (respectively right) BiHom-Lie algebra and
we define $\{ \cdot , \cdot \}:L\times L\rightarrow L$ by  $\{x, y\}:=[y, x]$, then $(L, \{ \cdot , \cdot \}, \beta , \alpha  )$ is a
right (respectively left) BiHom-Lie algebra.
\end{remark}
\begin{remark}
If $(L, [\cdot ,\cdot ])$ is a Lie algebra and $\alpha , \beta :L\rightarrow L$ are two commuting
morphisms of Lie algebras, and we define the map $\{ \cdot , \cdot \}:L\times L\rightarrow L$ by
$\{x, y\}=[\alpha (x), \beta (y)]$, then $(L, \{ \cdot , \cdot \}, \alpha , \beta )$ is a left and right
BiHom-Lie algebra, denoted by $L_{(\alpha , \beta )}$ and
called the Yau twist of $L$.
\end{remark}
\begin{remark}
More generally, let $(L, [\cdot ,\cdot ], \alpha , \beta )$ be a left (respectively right) BiHom-Lie algebra and
$\alpha ', \beta ':L\rightarrow L$ two morphisms of left (respectively right) BiHom-Lie algebras such that any two of the
maps $\alpha , \alpha ', \beta , \beta '$ commute. If we define the map $\{ \cdot , \cdot \}:L\times L\rightarrow L$,
$\{x, y\}=[\alpha '(x), \beta '(y)]$, then $(L, \{ \cdot , \cdot \}, \alpha \circ \alpha ', \beta \circ \beta ')$
is a left (respectively right)
BiHom-Lie algebra.
\end{remark}
\begin{proposition}\label{equivBHLie}
Let $L$ be a linear space, $[\cdot ,\cdot  ]:L\times L\rightarrow L$ a bilinear map, $\alpha , \beta :L\rightarrow L$
two commuting linear maps such that $\alpha ([x, y])=[\alpha (x), \alpha (y)]$ and
$\beta ([x, y])=[\beta (x), \beta (y)]$ for all $x, y\in L$, and the BiHom-skew-symmetry condition
(\ref{BHskewsym}) is satisfied. Assume that $\alpha $ and $\beta $ are bijective. Then
the following three conditions are equivalent: \\
(i) $(L, [\cdot ,\cdot ], \alpha , \beta )$ is a BiHom-Lie algebra; \\
(ii) $(L, [\cdot ,\cdot ], \alpha , \beta )$
 is a left BiHom-Lie algebra;\\
(iii) $(L, [\cdot ,\cdot ], \alpha , \beta )$
is a right BiHom-Lie algebra.
\end{proposition}
\begin{proof}
We  first prove that the BiHom-Jacobi condition
\begin{eqnarray*}
&&[\beta ^2(x), [\beta (y), \alpha (z)]]+[\beta ^2(y), [\beta (z), \alpha (x)]]+[\beta ^2(z), [\beta (x), \alpha (y)]]=0
\end{eqnarray*}
is equivalent to (\ref{leftBHleibniz}). The BiHom-Jacobi condition may be rewritten in one of the following equivalent forms:
\begin{eqnarray*}
&&[\beta ^2(x), \beta ([y, \alpha \beta ^{-1}(z)])]+[\beta ^2(y), \beta ([z, \alpha \beta ^{-1}(x)])]+
[\beta ^2(z), \beta ([x, \alpha \beta ^{-1}(y)])]=0,
\end{eqnarray*}
\begin{eqnarray*}
&&[\beta (x), [y, \alpha \beta ^{-1}(z)]]+[\beta (y), [z, \alpha \beta ^{-1}(x)]]+
[\beta (z), [x, \alpha \beta ^{-1}(y)]]=0,
\end{eqnarray*}
\begin{eqnarray*}
&&[\beta (x), [y, \alpha \beta ^{-1}(z)]]+[\beta (y), [z, \alpha \beta ^{-1}(x)]]+
[\beta (z), \alpha ([\alpha ^{-1}(x), \beta ^{-1}(y)])]=0,
\end{eqnarray*}
or, by applying the BiHom-skew-symmetry condition to the third term in the last equation,
\begin{eqnarray*}
&&[\beta (x), [y, \alpha \beta ^{-1}(z)]]+[\beta (y), [z, \alpha \beta ^{-1}(x)]]-
[\beta ([\alpha ^{-1}(x), \beta ^{-1}(y)]), \alpha (z)]=0,
\end{eqnarray*}
which, by replacing $x$ with $\alpha (x)$, is equivalent to
\begin{eqnarray*}
&&[\alpha \beta (x), [y, \alpha \beta ^{-1}(z)]]+[\beta (y), [z, \alpha ^2\beta ^{-1}(x)]]-
[[\beta (x), y], \alpha (z)]=0,
\end{eqnarray*}
and this relation, by replacing $z$ with $\alpha ^{-1}\beta (z)$, may be rewritten as
\begin{eqnarray*}
&&[\alpha \beta (x), [y, z]]+[\beta (y), [\alpha ^{-1}\beta (z), \alpha ^2\beta ^{-1}(x)]]-
[[\beta (x), y], \beta (z)]=0,
\end{eqnarray*}
or equivalently
\begin{eqnarray*}
&&[\alpha \beta (x), [y, z]]+[\beta (y), [\beta (\alpha ^{-1}(z)), \alpha (\alpha \beta ^{-1}(x))]]-
[[\beta (x), y], \beta (z)]=0,
\end{eqnarray*}
which, by applying the BiHom-skew-symmetry condition to the second bracket in the second term, may be transformed as
\begin{eqnarray*}
&&[\alpha \beta (x), [y, z]]-[\beta (y), [\alpha (x), z]]-
[[\beta (x), y], \beta (z)]=0,
\end{eqnarray*}
and this is obviously equivalent to (\ref{leftBHleibniz}).

Now we prove that (\ref{leftBHleibniz}) is equivalent to (\ref{rightBHleibniz}). We begin with
(\ref{leftBHleibniz}), rewritten as
\begin{eqnarray*}
&&[\alpha \beta (x), [y, z]]=
[[\beta (x), y], \beta (z)]+ [\beta (y), \alpha ([x, \alpha ^{-1}(z)])],
\end{eqnarray*}
which, by applying the BiHom-skew-symmetry condition to the second term in the right hand side, becomes
\begin{eqnarray*}
&&[\alpha \beta (x), [y, z]]=
[[\beta (x), y], \beta (z)]-[\beta ([x, \alpha ^{-1}(z)]), \alpha (y)],
\end{eqnarray*}
or equivalently
\begin{eqnarray*}
&&[\alpha \beta (x), [y, z]]=
[[\beta (x), y], \beta (z)]-[[\beta (x), \alpha ^{-1}\beta (z)], \alpha (y)],
\end{eqnarray*}
which, by replacing $x$ with $\beta ^{-1}(x)$ and $z$ with $\alpha (z)$,  may be rewritten as
\begin{eqnarray*}
&&[\alpha (x), [y, \alpha (z)]]=
[[x, y], \alpha \beta (z)]-[[x, \beta (z)], \alpha (y)],
\end{eqnarray*}
and this is obviously equivalent to (\ref{rightBHleibniz}).
\end{proof}
\begin{remark} By using one more time the same type of calculations, one can easily prove that, under the
hypotheses of Proposition \ref{equivBHLie}, $(L, [\cdot ,\cdot ], \alpha , \beta )$ is a BiHom-Lie algebra if and only if,
for all $x, y, z\in L$, the following equation is satisfied:
\begin{eqnarray*}
&&[[\beta (x), \alpha (y)], \alpha ^2(z)]+[[\beta (y), \alpha (z)], \alpha ^2(x)]+
[[\beta (z), \alpha (x)], \alpha ^2(y)]=0.
\end{eqnarray*}
\end{remark}
\begin{proposition}\label{lrprelie}
Let $(L, [\cdot ,\cdot ], \alpha , \beta )$ be a left (respectively right) BiHom-Lie algebra and $R:L\rightarrow L$ be a
Rota-Baxter operator of weight $0$ such that $R\circ \alpha =\alpha \circ R$ and $R\circ \beta =\beta \circ R$.
Define a new operation on $L$ by
$x\cdot y:=[R(x), y]$
(respectively by
$x\cdot y:=[x, R(y)]$), for all $x, y\in L$.
Then $(L, \cdot , \alpha , \beta )$ is a left (respectively right) BiHom-pre-Lie algebra.
\end{proposition}
\begin{proof}
We prove the left-handed version and leave the right-handed one to the reader. It is obvious that
$\alpha (x\cdot y)=\alpha (x)\cdot \alpha (y)$ and $\beta (x\cdot y)=\beta (x)\cdot \beta (y)$,
so we only need to prove the relation (\ref{lBHpL}) defining a left BiHom-pre-Lie algebra. We compute: \\[2mm]
${\;\;\;\;\;}$$\alpha \beta (x)\cdot (\alpha (y)\cdot z)-(\beta (x)\cdot \alpha (y))\cdot \beta (z)$
\begin{eqnarray*}
&=&\alpha \beta (x)\cdot [R(\alpha (y)), z]-[R(\beta (x)), \alpha (y)]\cdot \beta (z)\\
&=&[R(\alpha \beta (x)), [R(\alpha (y)), z]]-[R([R(\beta (x)), \alpha (y)]), \beta (z)]\\
&\overset{(\ref{RBrel})}{=}&[R(\alpha \beta (x)), [R(\alpha (y)), z]]-[[R(\beta (x)), R(\alpha (y))], \beta (z)]+
[R([\beta (x), R(\alpha (y))]), \beta (z)]\\
&=&[R(\alpha \beta (x)), [R(\alpha (y)), z]]-[[\beta (R(x)), R(\alpha (y))], \beta (z)]+
[R([\beta (x), R(\alpha (y))]), \beta (z)]\\
&\overset{(\ref{leftBHleibniz})}{=}&[R(\alpha \beta (x)), [R(\alpha (y)), z]]-[\alpha \beta (R(x)), [R(\alpha (y)), z]]+
[\alpha \beta (R(y)), [\alpha (R(x)), z]]\\
&&+[R([\beta (x), R(\alpha (y))]), \beta (z)]\\
&=&[\alpha \beta (R(y)), [\alpha (R(x)), z]]+[R([\beta (x), \alpha (R(y))]), \beta (z)]\\
&\overset{(\ref{BHskewsym})}{=}&[\alpha \beta (R(y)), [\alpha (R(x)), z]]-[R([\beta (R(y)), \alpha (x)]), \beta (z)]\\
&=&[R(\alpha \beta (y)), [R(\alpha (x)), z]]-[R([R(\beta (y)), \alpha (x)]), \beta (z)]\\
&=&\alpha \beta (y)\cdot (\alpha (x)\cdot z)-(\beta (y)\cdot \alpha (x))\cdot \beta (z),
\end{eqnarray*}
finishing the proof.
\end{proof}
\begin{lemma}
\label{Lemma:preR}We consider a 4-tuple $\left( L,\left[\cdot , \cdot \right]
,\alpha ,\beta \right) $, where $L$ is a linear space,
$\alpha , \beta :L\rightarrow L$ are linear maps and $\left[\cdot , \cdot \right]
:L\times L\rightarrow L$ is a bilinear map. Let also $\lambda \in \Bbbk $ be a fixed scalar
and  $R:L\rightarrow L$ be a linear map such that
\begin{equation}
R\circ \alpha =\alpha \circ R\text{ and }R\circ \beta =\beta \circ R.
\label{Ralpha}
\end{equation}%
Define a new multiplication on $L$ by
\begin{equation*}
\{x,y\}=[R(x),y]+[x,R(y)]+\lambda \lbrack x,y],\;\;\;\forall \;x,y\in L.
\end{equation*}%
Then

\begin{itemize}
\item[i)] If $\alpha $ and $\beta $ satisfy%
\begin{equation}
\alpha (\left[ x,y\right] )=\left[ \alpha \left( x\right) ,\alpha \left(
y\right) \right] \;\;\text{ and }\;\;\beta (\left[ x,y\right] )=\left[ \beta
\left( x\right) ,\beta \left( y\right) \right] ,\;\;\;\forall \;x,y\in L,
\label{alphapar}
\end{equation}%
then they also satisfy%
\begin{equation*}
\alpha (\left\{ x,y\right\} )=\left\{ \alpha \left( x\right) ,\alpha \left(
y\right) \right\} \;\;\text{ and }\;\;\beta (\left\{ x,y\right\} )=\left\{
\beta \left( x\right) ,\beta \left( y\right) \right\}, \;\;\;\forall
\;x,y\in L.
\end{equation*}

\item[ii)] If $\alpha $ and $\beta $ satisfy%
\begin{equation}
\left[ \beta \left( x\right) ,\alpha \left( y\right) \right] =-\left[ \beta
\left( y\right) ,\alpha \left( x\right) \right] ,\;\;\;\;\;\;\forall
\;x,y\in L,  \label{skew}
\end{equation}%
then they also satisfy%
\begin{equation*}
\left\{ \beta \left( x\right) ,\alpha \left( y\right) \right\} =-\left\{
\beta \left( y\right) ,\alpha \left( x\right) \right\}, \;\;\;\forall
\;x,y\in L.
\end{equation*}

\item[iii)] If $R$ satisfies
\begin{eqnarray}
\lbrack R(x),R(y)] &=&R([R(x),y]+[x,R(y)]+\lambda \lbrack
x,y]),\;\;\;\forall \;x,y\in L,  \label{Rota}
\end{eqnarray}%
then%
\begin{equation}
R\left( \{x,y\}\right) =[R(x),R(y)],\;\;\;\forall \;x,y\in L.  \label{Rgraf}
\end{equation}
\end{itemize}
\end{lemma}
\begin{proof}
i) We compute
\begin{eqnarray*}
\left\{ \alpha \left( x\right) ,\alpha \left( y\right) \right\} &=&[R(\alpha
\left( x\right) ),\alpha \left( y\right) ]+[\alpha \left( x\right) ,R(\alpha
\left( y\right) )]+\lambda \lbrack \alpha \left( x\right) ,\alpha \left(
y\right) ] \\
&\overset{\left( \ref{Ralpha}\right) }{=}&[\alpha (R\left( x\right) ),\alpha
\left( y\right) ]+[\alpha \left( x\right) ,\alpha (R\left( y\right)
)]+\lambda \lbrack \alpha \left( x\right) ,\alpha \left( y\right) ] \\
&\overset{\left( \ref{alphapar}\right) }{=}&\alpha (\left\{ x,y\right\} ).
\end{eqnarray*}

ii) We compute
\begin{eqnarray*}
\left\{ \beta \left( x\right) ,\alpha \left( y\right) \right\} &=&[R\left(
\beta (x)\right) ,\alpha \left( y\right) ]+[\beta \left( x\right) ,R(\alpha
\left( y\right) )]+\lambda \lbrack \beta \left( x\right) ,\alpha \left(
y\right) ] \\
&\overset{\left( \ref{Ralpha}\right) }{=}&[\beta \left( R(x)\right) ,\alpha
\left( y\right) ]+\left[ \beta \left( x\right) ,\alpha (R\left( y\right) )%
\right] +\lambda \lbrack \beta \left( x\right) ,\alpha \left( y\right) ] \\
&\overset{\left( \ref{skew}\right) }{=}&-[\beta \left( y\right) ,\alpha
\left( R(x)\right) ]-\left[ \beta \left( R\left( y\right) \right) ,\alpha (x)%
\right] -\lambda \left[ \beta \left( y\right) ,\alpha \left( x\right) \right]
\\
&\overset{\left( \ref{Ralpha}\right) }{=}&-[\beta \left( y\right) ,R(\alpha
\left( x\right) )]-[R\left( \beta (y)\right) ,\alpha \left( x\right)
]-\lambda \left[ \beta \left( y\right) ,\alpha \left( x\right) \right] \\
&=&-\left\{ \beta \left( y\right) ,\alpha \left( x\right) \right\} .
\end{eqnarray*}

iii) It is obvious by (\ref{Rota}) and the definition of $\{ \cdot , \cdot \}$.
\end{proof}

\begin{proposition}
\label{Proposition:BiHLie}Let $\left( L,\left[\cdot , \cdot \right]
,\alpha ,\beta \right) $ be a BiHom-Lie
algebra, $\lambda \in \Bbbk $ a fixed scalar and
$R:L\rightarrow L$ be a Rota-Baxter operator of weight $\lambda $ satisfying $R\circ \alpha =\alpha \circ R$ and
$R\circ \beta =\beta \circ R$ (i.e. $%
\left( \ref{Ralpha}\right) $ and $\left( \ref{Rota}\right) $ hold). Define a
new multiplication on $L$ by
\begin{equation*}
\{x,y\}=[R(x),y]+[x,R(y)]+\lambda \lbrack x,y],\;\;\;\forall \;x,y\in L.
\end{equation*}%
Then $(L,\left\{ {\cdot , \cdot }\right\} ,\alpha ,\beta )$ is a BiHom-Lie algebra.
\end{proposition}
\begin{proof}
In view of Lemma \ref{Lemma:preR}, we only need to prove that $\left\{
{\cdot , \cdot }\right\} $ satisfies the BiHom-Jacobi condition. Note also that, by Lemma %
\ref{Lemma:preR} again, $\left( \ref{Rgraf}\right) $ holds. We compute, for $x, y, z, t\in L$,
\begin{eqnarray*}
&&\left\{ \beta \left( x\right) ,\alpha \left( z\right) \right\} =[R(\beta
\left( x\right) ),\alpha \left( z\right) ]+[\beta \left( x\right) ,R(\alpha
\left( z\right) )]+\lambda \lbrack \beta \left( x\right) ,\alpha \left(
z\right) ],\\
&&\left\{ R(\beta \left( y\right)) ,\alpha \left( z\right) \right\} =[R(R(\beta
\left( y\right))),\alpha \left( z\right) ]+[R(\beta \left( y\right)) ,R(\alpha
\left( z\right) )]+\lambda \lbrack R(\beta \left( y\right)) ,\alpha \left(
z\right) ], \\
&&\left\{ \beta ^{2}\left( x\right) ,t\right\} =[R(\beta ^{2}(x)),t]+[\beta
^{2}(x),R(t)]+\lambda \lbrack \beta ^{2}(x),t].
\end{eqnarray*}
So, we obtain \\[2mm]
${\;\;\;}$
$\left\{ \beta ^{2}\left( x\right) ,\left\{ \beta \left( y\right) ,\alpha
\left( z\right) \right\} \right\}$
\begin{eqnarray*}
&=&[R(\beta ^{2}(x)),\left\{ \beta \left(
y\right) ,\alpha \left( z\right) \right\} ]+[\beta ^{2}(x),R(\left\{ \beta
\left( y\right) ,\alpha \left( z\right) \right\} )]+\lambda \lbrack \beta
^{2}(x),\left\{ \beta \left( y\right) ,\alpha \left( z\right) \right\} ] \\
&\overset{\left( \ref{Rgraf}\right) }{=}&[R(\beta ^{2}(x)),\left\{ \beta
\left( y\right) ,\alpha \left( z\right) \right\} ]+[\beta ^{2}(x),\left[
R(\beta \left( y\right)) ,R(\alpha \left( z\right)) \right] ]+\lambda \lbrack
\beta ^{2}(x),\left\{ \beta \left( y\right) ,\alpha \left( z\right) \right\} ]
\\
&=&[R(\beta ^{2}\left( x\right) ),[R(\beta \left( y\right) ),\alpha \left(
z\right) ]]+[R(\beta ^{2}\left( x\right) ),[\beta \left( y\right) ,R(\alpha
\left( z\right) )]]\\
&&
+[R(\beta ^{2}\left( x\right) ),\lambda \lbrack \beta
\left( y\right) ,\alpha \left( z\right) ]]+
[\beta ^{2}(x),\left[ R(\beta \left( y\right)) ,R(\alpha \left( z\right)) %
\right] ]\\
&&+\lambda \lbrack \beta ^{2}\left( x\right) ,[R(\beta \left( y\right)
),\alpha \left( z\right) ]]+\lambda \lbrack \beta ^{2}\left( x\right)
,[\beta \left( y\right) ,R(\alpha \left( z\right) )]]\\
&&+\lambda \lbrack \beta
^{2}\left( x\right) ,\lambda \lbrack \beta \left( y\right) ,\alpha \left(
z\right) ]],
\end{eqnarray*}%
and hence\\[2mm]
${\;\;\;\;\;\;}$
$\left\{ \beta ^{2}\left( y\right) ,\left\{ \beta \left( z\right) ,\alpha
\left( x\right) \right\} \right\}$
\begin{eqnarray*}
&=&[R(\beta ^{2}\left( y\right) ),[R(\beta
\left( z\right) ),\alpha \left( x\right) ]]+[R(\beta ^{2}\left( y\right)
),[\beta \left( z\right) ,R(\alpha \left( x\right) )]]\\
&&+[R(\beta ^{2}\left(
y\right) ),\lambda \lbrack \beta \left( z\right) ,\alpha \left( x\right) ]]+
[\beta ^{2}(y),\left[ R(\beta \left( z\right)) ,R(\alpha \left( x\right)) %
\right] ] \\
&&+\lambda \lbrack \beta ^{2}\left( y\right) ,[R(\beta \left( z\right)
),\alpha \left( x\right) ]]+\lambda \lbrack \beta ^{2}\left( y\right)
,[\beta \left( z\right) ,R(\alpha \left( x\right) )]]\\
&&+\lambda \lbrack \beta
^{2}\left( y\right) ,\lambda \lbrack \beta \left( z\right) ,\alpha \left(
x\right) ]],
\end{eqnarray*}%
${\;\;\;\;\;\;}$
$\left\{ \beta ^{2}\left( z\right) ,\left\{ \beta \left( x\right) ,\alpha
\left( y\right) \right\} \right\}$
\begin{eqnarray*}
&=&[R(\beta ^{2}\left( z\right) ),[R(\beta
\left( x\right) ),\alpha \left( y\right) ]]+[R(\beta ^{2}\left( z\right)
),[\beta \left( x\right) ,R(\alpha \left( y\right) )]]\\
&&+[R(\beta ^{2}\left(
z\right) ),\lambda \lbrack \beta \left( x\right) ,\alpha \left( y\right) ]]+
[\beta ^{2}(z),\left[ R(\beta \left( x\right)) ,R(\alpha \left( y\right)) %
\right] ]+ \\
&&+\lambda \lbrack \beta ^{2}\left( z\right) ,[R(\beta \left( x\right)
),\alpha \left( y\right) ]]+\lambda \lbrack \beta ^{2}\left( z\right)
,[\beta \left( x\right) ,R(\alpha \left( y\right) )]]\\
&&+\lambda \lbrack \beta
^{2}\left( z\right) ,\lambda \lbrack \beta \left( x\right) ,\alpha \left(
y\right) ]].
\end{eqnarray*}%
Thus, we get \\[2mm]
${\;}$
$\left\{ \beta ^{2}\left( x\right) ,\left\{ \beta \left( y\right) ,\alpha
\left( z\right) \right\} \right\} +\left\{ \beta ^{2}\left( y\right)
,\left\{ \beta \left( z\right) ,\alpha \left( x\right) \right\} \right\}
+\left\{ \beta ^{2}\left( z\right) ,\left\{ \beta \left( x\right) ,\alpha
\left( y\right) \right\} \right\}$
\begin{eqnarray*}
&=&[R(\beta ^{2}\left( x\right) ),[R(\beta \left( y\right) ),\alpha \left(
z\right) ]]+[R(\beta ^{2}\left( x\right) ),[\beta \left( y\right) ,R(\alpha
\left( z\right) )]]\\
&&+[R(\beta ^{2}\left( x\right) ),\lambda \lbrack \beta
\left( y\right) ,\alpha \left( z\right) ]]+
[\beta ^{2}(x),\left[ R(\beta \left( y\right)) ,R(\alpha \left( z\right)) %
\right] ]\\
&&+\lambda \lbrack \beta ^{2}\left( x\right) ,[R(\beta \left( y\right)
),\alpha \left( z\right) ]]+\lambda \lbrack \beta ^{2}\left( x\right)
,[\beta \left( y\right) ,R(\alpha \left( z\right) )]]+\lambda \lbrack \beta
^{2}\left( x\right) ,\lambda \lbrack \beta \left( y\right) ,\alpha \left(
z\right) ]]\\
&&+[R(\beta ^{2}\left( y\right) ),[R(\beta \left( z\right) ),\alpha \left(
x\right) ]]+[R(\beta ^{2}\left( y\right) ),[\beta \left( z\right) ,R(\alpha
\left( x\right) )]]\\
&&+[R(\beta ^{2}\left( y\right) ),\lambda \lbrack \beta
\left( z\right) ,\alpha \left( x\right) ]]+
[\beta ^{2}(y),\left[ R(\beta \left( z\right)) ,R(\alpha \left( x\right)) %
\right] ] \\
&&+\lambda \lbrack \beta ^{2}\left( y\right) ,[R(\beta \left( z\right)
),\alpha \left( x\right) ]]+\lambda \lbrack \beta ^{2}\left( y\right)
,[\beta \left( z\right) ,R(\alpha \left( x\right) )]]+\lambda \lbrack \beta
^{2}\left( y\right) ,\lambda \lbrack \beta \left( z\right) ,\alpha \left(
x\right) ]]\\
&&+[R(\beta ^{2}\left( z\right) ),[R(\beta \left( x\right) ),\alpha \left(
y\right) ]]+[R(\beta ^{2}\left( z\right) ),[\beta \left( x\right) ,R(\alpha
\left( y\right) )]]\\
&&+[R(\beta ^{2}\left( z\right) ),\lambda \lbrack \beta
\left( x\right) ,\alpha \left( y\right) ]]+
[\beta ^{2}(z),\left[ R(\beta \left( x\right)) ,R(\alpha \left( y\right)) %
\right] ]+ \\
&&+\lambda \lbrack \beta ^{2}\left( z\right) ,[R(\beta \left( x\right)
),\alpha \left( y\right) ]]+\lambda \lbrack \beta ^{2}\left( z\right)
,[\beta \left( x\right) ,R(\alpha \left( y\right) )]]+\lambda \lbrack \beta
^{2}\left( z\right) ,\lambda \lbrack \beta \left( x\right) ,\alpha \left(
y\right) ]],
\end{eqnarray*}
i.e., in view of  (\ref{Ralpha}), \\[2mm]
${\;}$
$\left\{ \beta ^{2}\left( x\right) ,\left\{ \beta \left( y\right) ,\alpha
\left( z\right) \right\} \right\} +\left\{ \beta ^{2}\left( y\right)
,\left\{ \beta \left( z\right) ,\alpha \left( x\right) \right\} \right\}
+\left\{ \beta ^{2}\left( z\right) ,\left\{ \beta \left( x\right) ,\alpha
\left( y\right) \right\} \right\}$
\begin{eqnarray*}
&=&[\beta ^{2}(R\left( x\right) ),[\beta (R(y)),\alpha \left(
z\right) ]]+[\beta ^{2}\left( R(x\right) ),[\beta \left( y\right) ,\alpha
(R(z))]]+[\beta ^{2}(z),\left[ \beta \left( R\left( x\right)
\right) ,\alpha \left( R\left( y\right) \right) \right] ] \\
&&+[\beta ^{2}\left( R(y\right) ),[\beta \left( R(z\right) ),\alpha \left(
x\right) ]]+[\beta ^{2}(y),\left[\beta \left( R(z)\right) ,\alpha \left(
R(x)\right) \right] ]\\
&&+[\beta ^{2}\left( R(y\right) ),[\beta \left( z\right)
,\alpha \left( R(x\right) )]]
+[\beta ^{2}(x),\left[ \beta \left( R\left( y\right) \right) ,\alpha \left(
R\left( z\right) \right) \right] ]\\
&&+[\beta ^{2}(R\left( z\right) ),[\beta
\left( R(x\right) ),\alpha \left( y\right) ]]+[\beta ^{2}\left( R(z\right)
),[\beta \left( x\right) ,\alpha \left( R(y\right) )]] \\
&&+\lambda \lbrack \beta ^{2}\left( x\right) ,[\beta \left( R(y\right)
),\alpha \left( z\right) ]]+[\beta ^{2}\left( R(y\right) ),\lambda \lbrack
\beta \left( z\right) ,\alpha \left( x\right) ]]+\lambda \lbrack \beta
^{2}\left( z\right) ,[\beta \left( x\right) ,\alpha (R(y))]] \\
&&+\lambda \lbrack \beta ^{2}\left( x\right) ,[\beta \left( y\right) ,\alpha
\left( R(z\right) )]]+\lambda \lbrack \beta ^{2}\left( y\right) ,[\beta
\left( R(z\right) ),\alpha \left( x\right) ]]+[\beta ^{2}\left( R(z\right)
),\lambda \lbrack \beta \left( x\right) ,\alpha \left( y\right) ]] \\
&&+[\beta ^{2}\left( R(x\right) ),\lambda \lbrack \beta \left( y\right)
,\alpha \left( z\right) ]]+\lambda \lbrack \beta ^{2}\left( y\right) ,[\beta
\left( z\right) ,\alpha \left( R(x\right) )]]+\lambda \lbrack \beta
^{2}\left( z\right) ,[\beta \left( R(x\right) ),\alpha \left( y\right) ]] \\
&&+\lambda \lbrack \beta ^{2}\left( x\right) ,\lambda \lbrack \beta \left(
y\right) ,\alpha \left( z\right) ]]+\lambda \lbrack \beta ^{2}\left( y\right)
,\lambda \lbrack \beta \left( z\right) ,\alpha \left( x\right) ]]+\lambda
\lbrack \beta ^{2}\left( z\right) ,\lambda \lbrack \beta \left( x\right)
,\alpha \left( y\right) ]]\\
&=&\circlearrowleft _{R(x), R(y), z}[\beta ^2(R(x)), [\beta (R(y)), \alpha (z)]]+
\circlearrowleft _{R(y), R(z), x}[\beta ^2(R(y)), [\beta (R(z)), \alpha (x)]]\\
&&+\circlearrowleft _{R(z), R(x), y}[\beta ^2(R(z)), [\beta (R(x)), \alpha (y)]]+
\lambda \circlearrowleft _{x, R(y), z}[\beta ^2(x), [\beta (R(y)), \alpha (z)]]\\
&&+\lambda \circlearrowleft _{y, R(z), x}[\beta ^2(y), [\beta (R(z)), \alpha (x)]]+
\lambda \circlearrowleft _{z, R(x), y}[\beta ^2(z), [\beta (R(x)), \alpha (y)]]\\
&&+\lambda ^2\circlearrowleft _{x, y, z}[\beta ^2(x), [\beta (y), \alpha (z)]]\\
&=&0,
\end{eqnarray*}
where the last equality follows by applying 7 times the BiHom-Jacobi condition for $[\cdot ,\cdot  ]$.
\end{proof}

By taking $\alpha =\beta =id_{L}$ in  Proposition \ref{Proposition:BiHLie}, we obtain the following well-known result:

\begin{corollary} (\cite{bai})
Let $(L,[\cdot ,\cdot ])$ be a Lie algebra, $\lambda \in \Bbbk $ a fixed scalar and $R:L\rightarrow L$ a Rota-Baxter operator
of weight $\lambda $.
Define a new multiplication on $L$ by
\begin{equation*}
\{x,y\}=[R(x),y]+[x,R(y)]+\lambda \lbrack x,y],\;\;\;\forall \;x,y\in L.
\end{equation*}%
Then $(L,\{ \cdot , \cdot \})$ is also a Lie algebra (and of course we have $%
R(\{x,y\})=[R(x),R(y)]$).
\end{corollary}
\begin{proposition}
\label{Proposition:Leibniz} Let $(L,[\cdot ,\cdot ],\alpha ,\beta )$ be a left (respectively right)
BiHom-Leibniz algebra, $\lambda \in \Bbbk $ a fixed scalar and
$R:L\rightarrow L$ a Rota-Baxter operator of weight $\lambda $ such that
$R\circ \alpha =\alpha \circ R$ and $R\circ \beta =\beta \circ R$.  Define a
new multiplication on $L$ by
\begin{equation*}
\{x,y\}=[R(x),y]+[x,R(y)]+\lambda \lbrack x,y],\;\;\;\forall \;x,y\in L.
\end{equation*}%
Then $(L,\left\{\cdot , \cdot \right\} ,\alpha ,\beta )$ is a left (respectively right) BiHom-Leibniz algebra.
\end{proposition}
\begin{proof} We only prove the left-handed version, the right-handed one is similar
and left to the reader.
In view of Lemma \ref{Lemma:preR}, we only need to show that $\left\{
{\cdot , \cdot }\right\} $ satisfies (\ref{leftBHleibniz}). Note also that, by
Lemma \ref{Lemma:preR} again, $\left( \ref{Rgraf}\right) $ holds. We compute, for all
$x, y, z, t\in L$,
\begin{eqnarray*}
\{\alpha \beta (x),t\}&=&[R(\alpha \beta (x)),t]+[\alpha \beta
(x),R(t)]+\lambda \lbrack \alpha \beta (x),t], \\
\{\alpha \beta (x),\{y,z\}\} &=&[R(\alpha \beta (x)),[R(y),z]]+[R(\alpha
\beta (x)),[y,R(z)]]+[R(\alpha \beta (x)),\lambda \lbrack y,z]]\\
&&+[\alpha \beta (x),R(\{y,z\})]+\lambda \lbrack \alpha \beta (x),[R(y),z]]\\
&&+\lambda \lbrack \alpha \beta
(x),[y,R(z)]]+\lambda \lbrack \alpha \beta (x),\lambda \lbrack y,z]] \\
&\overset{\left( \ref{Rgraf}\right) }{=}&[R(\alpha \beta
(x)),[R(y),z]]+[R(\alpha \beta (x)),[y,R(z)]]+[R(\alpha \beta (x)),\lambda
\lbrack y,z]] \\
&&+[\alpha \beta (x),[R(y),R(z)]]+\lambda \lbrack \alpha \beta
(x),[R(y),z]]\\
&&+\lambda \lbrack \alpha \beta (x),[y,R(z)]]+\lambda \lbrack
\alpha \beta (x),\lambda \lbrack y,z]],
\end{eqnarray*}%
which, by using the fact that $R\circ \alpha =\alpha \circ R$ and $R\circ \beta =\beta \circ R$, may be written as
\begin{eqnarray*}
\{\alpha \beta (x),\{y,z\}\} &=&[\alpha \beta (R(x)),[R(y),z]]
+[\alpha \beta (R(x)),[y,R(z)]]
+[\alpha \beta (R(x)),\lambda \lbrack y,z]] \\
&&+[\alpha \beta (x),[R(y),R(z)]]
+\lambda \lbrack \alpha \beta (x),[R(y),z]] \\
&&+\lambda \lbrack \alpha \beta (x),[y,R(z)]]
+\lambda \lbrack \alpha \beta (x),\lambda \lbrack y,z]].
\end{eqnarray*}%

On the other hand, we have
\begin{eqnarray*}
\left\{ \beta (x),y\right\} &=&[R(\beta (x)),y]+[\beta (x),R(y)]+\lambda
\lbrack \beta (x),y], \\
\left\{ t,\beta (z)\right\} &=&[R(t),\beta (z)]+[t,R(\beta (z))]+\lambda
\lbrack t,\beta (z)], \\
\left\{ \left\{ \beta (x),y\right\} ,\beta (z)\right\} &=&[R(\left\{ \beta
(x),y\right\} ),\beta (z)]+ [[R(\beta (x)),y],R(\beta (z))]\\
&&+[[\beta (x),R(y)],R(\beta (z))]+[\lambda
\lbrack \beta (x),y],R(\beta (z))] \\
&&+\lambda \lbrack \lbrack R(\beta (x)),y],\beta (z)]+\lambda \lbrack
\lbrack \beta (x),R(y)],\beta (z)]+\lambda \lbrack \lambda \lbrack \beta
(x),y],\beta (z)] \\
&\overset{\left( \ref{Rgraf}\right) }{=}&[[R(\beta (x)),R(y)],\beta (z)]+
[[R(\beta (x)),y],R(\beta (z))]\\
&&+[[\beta (x),R(y)],R(\beta (z))]+[\lambda
\lbrack \beta (x),y],R(\beta (z))] \\
&&+\lambda \lbrack \lbrack R(\beta (x)),y],\beta (z)]+\lambda \lbrack
\lbrack \beta (x),R(y)],\beta (z)]+\lambda \lbrack \lambda \lbrack \beta
(x),y],\beta (z)]
\end{eqnarray*}%
and%
\begin{eqnarray*}
\left\{ \beta (y),t\right\} &=&[R(\beta (y)),t]+[\beta (y),R(t)]+\lambda
\lbrack \beta (y),t], \\
\left\{ \alpha (x),z\right\} &=&[R(\alpha (x)),z]+[\alpha (x),R(z)]+\lambda
\lbrack \alpha (x),z], \\
\left\{ \beta (y),\left\{ \alpha (x),z\right\} \right\} &=&[R(\beta
(y)),[R(\alpha (x)),z]]+[R(\beta (y)),[\alpha (x),R(z)]]\\
&&+[R(\beta
(y)),\lambda \lbrack \alpha (x),z]]+
[\beta (y),R(\left\{ \alpha (x),z\right\} )] \\
&&+\lambda \lbrack \beta (y),[R(\alpha (x)),z]]+\lambda \lbrack \beta
(y),[\alpha (x),R(z)]]+\lambda \lbrack \beta (y),\lambda \lbrack \alpha
(x),z]] \\
&\overset{\left( \ref{Rgraf}\right) }{=}&[R(\beta (y)),[R(\alpha
(x)),z]]+[R(\beta (y)),[\alpha (x),R(z)]]\\
&&+[R(\beta (y)),\lambda \lbrack
\alpha (x),z]]+
[\beta (y),[R(\alpha (x)),R(z)]]\\
&&+\lambda \lbrack \beta (y),[R(\alpha
(x)),z]]+\lambda \lbrack \beta (y),[\alpha (x),R(z)]]+\lambda \lbrack \beta
(y),\lambda \lbrack \alpha (x),z]].
\end{eqnarray*}%
So \\[2mm]
${\;\;\;}$
$\left\{ \left\{ \beta (x),y\right\} ,\beta (z)\right\} +\left\{ \beta
(y),\left\{ \alpha (x),z\right\} \right\}$
\begin{eqnarray*}
&=&[[R(\beta (x)),R(y)],\beta (z)]
+[[R(\beta (x)),y],R(\beta (z))]+[[\beta (x),R(y)],R(\beta (z))]\\
&&+[\lambda
\lbrack \beta (x),y],R(\beta (z))]
+\lambda \lbrack \lbrack R(\beta (x)),y],\beta (z)]+\lambda \lbrack
\lbrack \beta (x),R(y)],\beta (z)]\\
&&+\lambda \lbrack \lambda \lbrack \beta
(x),y],\beta (z)]+
[R(\beta (y)),[R(\alpha (x)),z]]+[R(\beta (y)),[\alpha (x),R(z)]]\\
&&+[R(\beta
(y)),\lambda \lbrack \alpha (x),z]]+
[\beta (y),[R(\alpha (x)),R(z)]]+\lambda \lbrack \beta (y),[R(\alpha
(x)),z]]\\
&&+\lambda \lbrack \beta (y),[\alpha (x),R(z)]]+\lambda \lbrack \beta
(y),\lambda \lbrack \alpha (x),z]],
\end{eqnarray*}%
i.e., by using again the fact that $R\circ \alpha =\alpha \circ R$ and $R\circ \beta =\beta \circ R$, \\[2mm]
${\;\;\;\;\;\;\;\;\;\;}$
$\left\{ \left\{ \beta (x),y\right\} ,\beta (z)\right\} +\left\{ \beta
(y),\left\{ \alpha (x),z\right\} \right\}$
\begin{eqnarray*}
&=&[[\beta (R(x)),R(y)],\beta (z)]+[\beta (R(y)),[\alpha (R(x)),z]]+ \\
&&+[[\beta (R(x)),y],\beta (R(z))]+[\beta (y),[\alpha (R(x)),R(z)]] \\
&&+\lambda \lbrack \lbrack \beta (R(x)),y],\beta (z)]+\lambda \lbrack \beta
(y),[\alpha (R(x)),z]] \\
&&+[[\beta (x),R(y)],\beta (R(z))]+[\beta (R(y)),[\alpha (x),R(z)]] \\
&&+\lambda \lbrack \lbrack \beta (x),R(y)],\beta (z)]+[\beta (R(y)),\lambda
\lbrack \alpha (x),z]] \\
&&+[\lambda \lbrack \beta (x),y],\beta (R(z))]+\lambda \lbrack \beta
(y),[\alpha (x),R(z)]] \\
&&+\lambda \lbrack \lambda \lbrack \beta (x),y],\beta (z)]+\lambda \lbrack
\beta (y),\lambda \lbrack \alpha (x),z]].
\end{eqnarray*}
By using 7 times (\ref{leftBHleibniz}), we obtain $\{\alpha \beta
(x),\{y,z\}\}=\left\{ \left\{ \beta (x),y\right\} ,\beta (z)\right\}
+\left\{ \beta (y),\left\{ \alpha (x),z\right\} \right\}$.
\end{proof}
\begin{corollary} \label{3.16}
Let $(L,[\cdot ,\cdot ],\alpha ,\beta )$ be a left (respectively right) BiHom-Lie algebra,
$\lambda \in \Bbbk $ a fixed scalar and
$R:L\rightarrow L$ a Rota-Baxter operator of weight $\lambda $ such that
$R\circ \alpha =\alpha \circ R$ and $R\circ \beta =\beta \circ R$.  Define a
new multiplication on $L$ by
\begin{equation*}
\{x,y\}=[R(x),y]+[x,R(y)]+\lambda \lbrack x,y],\;\;\;\forall \;x,y\in L.
\end{equation*}%
Then $(L,\left\{\cdot , \cdot \right\} ,\alpha ,\beta )$ is a left (respectively right) BiHom-Lie algebra.
\end{corollary}
\begin{proof}
It follows by Proposition \ref{Proposition:Leibniz} in view
of ii) in Lemma \ref{Lemma:preR}.
\end{proof}
\section{BiHom-PostLie algebras}\label{sec6}
\setcounter{equation}{0}
The aim of this section is to prove that, in the case of bijective structure maps,
Propositions \ref{lrprelie} and \ref{Proposition:BiHLie} above admit a common generalization.
\begin{definition} (\cite{bai}, \cite{vallette}) A left PostLie algebra is a triple $(L, [\cdot , \cdot ], \circ )$, where $L$ is a linear space endowed with two bilinear operations $[\cdot , \cdot ]$ and
$\circ $ such that $(L, [\cdot , \cdot ])$ is a Lie algebra and the following relations are satisfied, for all $x, y, z\in L$:
\begin{eqnarray*}
&&z\circ (y\circ x)-y\circ (z\circ x)+(y\circ z)\circ x-(z\circ y)\circ x+[y, z]\circ x=0, \\
&&z\circ [x, y]-[z\circ x, y]-[x, z\circ y]=0.
\end{eqnarray*}
\end{definition}
\begin{definition} A left BiHom-PostLie algebra is a 5-tuple $(L, \{\cdot , \cdot \}, *, \alpha , \beta )$,
where $L$ is a linear space, $\{\cdot , \cdot \}$ and $*$ are bilinear operations on $L$, $\alpha , \beta :L\rightarrow L$
are two commuting linear maps such that $\alpha (\{x, y\})=\{\alpha (x), \alpha (y)\}$,
$\beta (\{x, y\})=\{\beta (x), \beta (y)\}$, $\alpha (x*y)=\alpha (x)*\alpha (y)$, $\beta (x*y)=\beta (x)*\beta (y)$,
for all $x, y\in L$, $(L, \{\cdot , \cdot \}, \alpha , \beta )$ is a left BiHom-Lie algebra and the following identities are
satisfied, for all $x, y, z\in L$:
\begin{eqnarray}
&&\alpha \beta (z)*(\alpha (y)* x)-\alpha \beta (y)* (\alpha (z)* x)+(\beta (y)* \alpha (z))*\beta (x)\nonumber \\
&&\;\;\;\;\;\;\;\;-(\beta (z)* \alpha (y))* \beta (x)+\{\beta (y), \alpha (z)\}* \beta (x)=0, \label{PostLie1}\\
&&\alpha \beta (z)* \{x, y\}-\{\beta (z)* x, \beta (y)\}-\{\beta (x), \alpha (z)* y\}=0.\label{PostLie2}
\end{eqnarray}
\end{definition}
\begin{proposition}
Let $(L, [\cdot , \cdot ], \circ )$ be a left PostLie algebra and $\alpha , \beta :L\rightarrow L$
two commuting linear maps such that $\alpha ([x, y])=[\alpha (x), \alpha (y)]$, $\beta ([x, y])=[\beta (x), \beta (y)]$,
$\alpha (x\circ y)=\alpha (x)\circ \alpha (y)$, $\beta (x\circ y)=\beta (x)\circ \beta (y)$, for all $x, y\in L$. Define two new
operations on $L$ by $\{x, y\}=[\alpha (x), \beta (y)]$ and $x*y=\alpha (x)\circ \beta (y)$, for all $x, y\in L$.
Then $(L, \{\cdot , \cdot \}, *, \alpha , \beta )$ is a left BiHom-PostLie algebra, called the Yau twist of  $(L, [\cdot , \cdot ], \circ )$.
\end{proposition}
\begin{remark}
Obviously, if the bracket $\{\cdot , \cdot \}$ in the definition of a left BiHom-PostLie algebra is trivial, then a left
BiHom-PostLie algebra is just a left BiHom-pre-Lie algebra.
\end{remark}

A proof for the next result may be found in \cite{adimi}.
\begin{proposition}\label{interm}
Let $(L, \{\cdot , \cdot \}, *, \alpha , \beta )$ be a left BiHom-PostLie algebra such that $\alpha $ and $\beta $ are
bijective. Define a new multiplication on $L$  by
\begin{eqnarray*}
&&<x, y>:=x*y-(\alpha ^{-1}\beta (y))*(\alpha \beta ^{-1}(x))+\{x, y\}, \;\;\; \forall \;\;x, y\in L.
\end{eqnarray*}
Then $(L, <\cdot , \cdot >, \alpha , \beta )$ is a BiHom-Lie algebra.
\end{proposition}

The next result is the BiHom analogue of Corollary 5.6 in \cite{bai}, we leave the proof to the reader.
\begin{proposition} \label{comgen}
Let $(L, [\cdot , \cdot ], \alpha , \beta )$ be a left BiHom-Lie algebra, $\lambda \in \Bbbk $ a fixed scalar and $R:L\rightarrow L$ a Rota-Baxter operator
of weight $\lambda $ such that $R\circ \alpha =\alpha \circ R$ and $R\circ \beta =\beta \circ R$. Define two new operations
on $L$ by
\begin{eqnarray*}
&&\{x, y\}:=\lambda [x, y] \;\;\;and\;\;\;x*y:=[R(x), y], \;\;\; \forall \;\;x, y\in L.
\end{eqnarray*}
Then $(L, \{\cdot , \cdot \}, *, \alpha , \beta )$ is a left BiHom-PostLie algebra.
\end{proposition}

Note that, by taking $\lambda =0$ in Proposition \ref{comgen}, we obtain Proposition \ref{lrprelie}. On the other hand,
if $\alpha $ and $\beta $ are bijective (and $\lambda $ is arbitrary), then Proposition
\ref{Proposition:BiHLie} (or, equivalently, Corollary \ref{3.16}) is a particular case of Proposition \ref{comgen}. Indeed,
to the left BiHom-PostLie algebra $(L, \{\cdot , \cdot \}, *, \alpha , \beta )$ obtained in Proposition \ref{comgen}
we apply Proposition \ref{interm}, and we obtain a BiHom-Lie algebra $(L, <\cdot , \cdot >, \alpha , \beta )$, where
the bracket $<\cdot , \cdot >$ is given by
\begin{eqnarray*}
<x, y>&=&x*y-(\alpha ^{-1}\beta (y))*(\alpha \beta ^{-1}(x))+\{x, y\}\\
&=&[R(x), y]-[R\alpha ^{-1}\beta (y), \alpha \beta ^{-1}(x)]+\lambda [x, y]\\
&=&[R(x), y]-[\beta \alpha ^{-1}R(y), \alpha \beta ^{-1}(x)]+\lambda [x, y]\\
&\overset{(\ref{BHskewsym})}{=}&[R(x), y]+[\beta \beta ^{-1}(x), \alpha \alpha ^{-1}R(y)]+\lambda [x, y]\\
&=&[R(x), y]+[x, R(y)]+\lambda [x, y],
\end{eqnarray*}
which is exactly the bracket in Proposition \ref{Proposition:BiHLie}.
\section{Conclusions}\label{sec7}
\setcounter{equation}{0}
This paper is a contribution to the study of Rota-Baxter operators on other types of algebras than associative and Lie algebras.
The starting point was the following question, formulated in our previous paper \cite{lmmp1}: if $(A, \cdot )$ is an algebra of a certain type, $R:A\rightarrow A$ is a Rota-Baxter operator of weight $\lambda $ and we
define a new multiplication on $A$ by $a\bullet b=R(a)\cdot b+a\cdot R(b)+\lambda a\cdot b$, for all $a, b\in A$, then under
what circumstances is $(A, \bullet )$ an algebra of the same type as $(A, \cdot )$? In \cite{lmmp1} we treated the case
of BiHom-associative algebras, while in this paper we investigated the case of BiHom-Lie algebras. The second aim of the paper
was to obtain a BiHom version of the result of Aguiar telling how to obtain a pre-Lie algebra from a Rota-Baxter operator of weight zero on a Lie algebra. It turned out that the BiHom version does not work for BiHom-Lie algebras, but for a new concept
we introduce in this paper, called  left BiHom-Lie algebra. It is an interesting phenomenon that the concepts of Lie algebra
and Hom-Lie algebra admit two different generalizations in the BiHom context, and we expect that the new concept of
left BiHom-Lie algebra, introduced here, will appear in many other situations. Actually, we expect that some results from
classical Lie theory may be extended to BiHom-Lie algebras, while others may be extended to left BiHom-Lie algebras.
For instance, there are indications that representation theory might work better for left BiHom-Lie algebras than for
BiHom-Lie algebras.

\begin{center}
ACKNOWLEDGEMENTS
\end{center}

This paper was written while Ling Liu was visiting the Institute of Mathematics of the Romanian Academy (IMAR),
supported by the NSF of China (Grant Nos. 11801515, 11601486); she would like to thank IMAR for hospitality.
Claudia Menini was a  member of the National Group for Algebraic
and Geometric Structures, and their Applications (GNSAGA-INdAM).

\bibliographystyle{amsplain}
\providecommand{\bysame}{\leavevmode\hbox to3em{\hrulefill}\thinspace}
\providecommand{\MR}{\relax\ifhmode\unskip\space\fi MR }
\providecommand{\MRhref}[2]{%
  \href{http://www.ams.org/mathscinet-getitem?mr=#1}{#2}
}
\providecommand{\href}[2]{#2}

\end{document}